\documentclass[11pt,a4paper]{article}
\usepackage{pgf}
\usepackage{tikz}
\usetikzlibrary{arrows,automata}
\usepackage{verbatim}
\usepackage{epsf,epsfig,amsfonts,amsgen,amstext,amsbsy,amsopn,amsthm,amssymb}
\usepackage{color}
\usepackage{mathrsfs,hyperref}
\usepackage{enumerate}
\usepackage{url}
\usepackage{enumitem}
\usetikzlibrary{patterns}
\usepackage{pythonhighlight}
\usepackage{listings}
\usepackage{amsmath}
\usepackage{makecell}
\usepackage{nicematrix}
\usepackage{array}
\usepackage{booktabs}
\usepackage{caption}
\definecolor{dkgreen}{rgb}{0,0.6,0}
\definecolor{gray}{rgb}{0.5,0.5,0.5}
\definecolor{mauve}{rgb}{0.58,0,0.82}
\usepackage{setspace}

\newtheorem{thm}{Theorem}[section]

\newtheorem{dfn}[thm]{Definition}

\newtheorem{claim}[thm]{Claim}

\newtheorem{lem}[thm]{Lemma}

\newtheorem{conj}[thm]{Conjecture}

\def\QED{\hfill \rule{7pt}{7pt}}

\newcommand{\turangraph}[2]{T_{#1,#2}}
\newcommand{\turannumber}[2]{t_{#1,#2}}

\makeatletter

\newcommand{\Rmnum}[1]{\expandafter\@slowromancap\romannumeral #1@}
\makeatother
\begin{document}
	\title{On the number of triangles in $K_4$-free graphs}
	
	\author{
    Jialin He\thanks{Department of Mathematics, Hong Kong University of Science and Technology,  Kowloon, Hong Kong. 
            Partially supported by Hong Kong RGC grant GRF 16308219 and Hong Kong RGC grant ECS 26304920. 
            Email: majlhe@ust.hk}\and
    Jie Ma\thanks{School of Mathematical Sciences, University of Science and Technology of China, Hefei, Anhui 230026, and Yau Mathematical Sciences Center, Tsinghua University, Beijing 100084, China.
			Research supported by National Key Research and Development Program of China 2023YFA1010201 and National Natural Science Foundation of China grant 12125106. 
            Email: jiema@ustc.edu.cn}\and
    Yan Wang\thanks{School of Mathematical Sciences, Shanghai Jiao Tong University, Shanghai 200240, China. 
            Supported by National Key R\&D Program of China under Grant No. 2022YFA1006400, National Natural Science Foundation of China (Nos. 12571376) and Shanghai Municipal Education Commission (No. 2024AIYB003). 
            Email: yan.w@sjtu.edu.cn}\and
    Chunlei Zu\thanks{Email: zucle@mail.ustc.edu.cn}
    }
	\date{}
	
	\maketitle

\begin{abstract}
Erd\H os asked whether for any \(n\)-vertex graph \(G\), the parameter 
\(p^*(G)=\min \sum_{i\ge 1} (|V(G_i)|-1)\) is at most \(\lfloor n^2/4\rfloor\), 
where the minimum is taken over all edge decompositions of \(G\) into edge-disjoint cliques \(G_i\).
In a restricted case (also conjectured independently by Erd\H{o}s), Gy\H ori and Keszegh 
[Combinatorica, 37(6) (2017), 1113--1124]
proved that \(p^*(G)\leq \lfloor n^2/4\rfloor\) for all \(K_4\)-free graphs \(G\).
Motivated by their proof approach, they conjectured that for any \(n\)-vertex \(K_4\)-free graph \(G\) with \(e\) edges,
and any greedy partition \(P\) of \(G\) of size \(r\),
the number of triangles in \(G\) is at least \(r(e-r(n-r))\).
If true, this would imply a stronger bound on \(p^*(G)\).
In this paper, we disprove their conjecture by constructing infinitely many counterexamples with arbitrarily large gap.
We further establish a corrected tight lower bound on the number of triangles in such graphs, which would recover the conjectured bound once some small counterexamples we identify are excluded.
\end{abstract}

\section{Introduction}
The {\it Tur\'an graph} $\turangraph{n}{k-1}$ denotes the complete balanced $(k-1)$-partite graph on $n$ vertices, and let $\turannumber{n}{k-1}$ be its number of edges.
A graph is \textit{$K_k$-free} if it contains no copy of the clique $K_k$ as a subgraph.
The celebrated Tur\'an's theorem~\cite{Turan}, a cornerstone of extremal graph theory, states that the Tur\'an graph $\turangraph{n}{k-1}$ is the unique $n$-vertex $K_k$-free graph with the maximum number of edges.
Exploring various interpretations and extensions of Tur\'an's theorem has long been one of the central themes in extremal graph theory.
An early result along this line, due to Erd\H{o}s, Goodman, and P\'osa~\cite{EGP}, shows that the edge set of every $n$-vertex graph can be decomposed into at most $\turannumber{n}{2}$ edge-disjoint triangles $K_3$ and individual edges.
Later, this was generalized by Bollob\'as~\cite{B1976}, who proved that for all $k\geq 3$, the edge set of every $n$-vertex graph can be decomposed into at most $\turannumber{n}{k-1}$ edge-disjoint cliques $K_k$ and individual edges.
It is clear that this result extends Tur\'an's theorem.
Another problem closely related to our study is the determination of the parameter $$p(G)=\min\sum\limits_{i\geq 1} |V(G_i)|$$ for any graph $G$,
where the minimum is taken over all edge decompositions of $G$ into edge-disjoint cliques $G_i$ for $i\geq 1$.
Chung~\cite{Chung}, Gy\H{o}ri and Kostochka~\cite{GK1979}, and Kahn~\cite{Kahn} independently proved that
$p(G)\le 2\turannumber{n}{2}$, with equality if and only if $G$ is the complete balanced bipartite graph $\turangraph{n}{2}$.

Erd\H{o}s (see \cite{Tuza}) later proposed to study the following enhanced variant of $p(G)$: 
\[
p^*(G)=\min \sum_{i\ge 1} \bigl(|V(G_i)|-1\bigr)
\]
for any graph $G$, where the minimum is taken over all edge decompositions of $G$ into edge-disjoint cliques $G_i$ for $i\geq 1$.
Clearly, $p^*(G)<p(G)$ holds for every graph $G$.
Erd\H{o}s posed the following challenging problem (see Problem~43 in~\cite{Tuza} and Conjecture~3 in~\cite{Gyori}): 
Does every $n$-vertex graph $G$ satisfy $p^*(G) \leq t_{n,2}$?
Recently, the first author, together with Balogh, Krueger, Nguyen, and Wigal~\cite{BHKNW}, proved an asymptotic version of this problem, showing that
$p^*(G) \leq (1+o(1))\turannumber{n}{2}$ holds for every $n$-vertex graph $G$, where $o(1)\to 0$ as $n\to \infty$.
A restricted case of Erd\H{o}s' problem was considered by Gy\H{o}ri~\cite{Gyori}, who estimated $p^*(G)$ for $K_4$-free graphs.
It turns out that this restricted version is equivalent to a problem of bounding edge-disjoint triangles in $K_4$-free graphs,
which was independently conjectured by Erd\H{o}s (see~\cite{H05}) and later resolved by Gy\H{o}ri and Keszegh~\cite{GK2017} in the following theorem.

\begin{thm}[Gy\H{o}ri and Keszegh~\cite{GK2017}]\label{Thm:original}
    Every $K_4$-free graph with $n$ vertices and $\turannumber{n}{2}+m$ edges contains at least $m$ edge-disjoint triangles.
\end{thm}

A crucial concept in their proof \cite{GK2017} is the following special partition of vertex-disjoint cliques.  
We call it {\it greedy} because it can be obtained by iteratively applying the following greedy procedure: at each step, select a largest clique in the remaining graph and then delete the vertices of this clique.

\begin{dfn}
A greedy partition $P$ of a graph $G$ is a partition of $V(G)$ into disjoint cliques $T_i$ for $i\geq 1$ such that $|T_i| \geq |T_{i+1}|$ for each $i\geq 1$ and, for each $\ell \geq 1$, the union of cliques with size at most $\ell$ induces a $K_{\ell+1}$-free subgraph.
The \textit{size} $r(P)$ of $P$ denotes the number of cliques in this partition.
\end{dfn}

Throughout, let $t(G)$ denote the number of triangles in a graph $G$, and let $t_e(G)$ denote the maximum number of edge-disjoint triangles in $G$.\footnote{When the graph $G$ is clear from context, we simply write $t$ and $t_e$.}  
A useful lemma of Huang and Shi~\cite{HS2014} relates these parameters via greedy partitions:
for any $K_4$-free graph $G$ and any greedy partition $P$ of $G$, we have 
\begin{equation}\label{equ:t_e}
t_e(G)\geq t(G)/r(P).
\end{equation}
Gy\H{o}ri and Keszegh~\cite{GK2017} employed an approach based on greedy partitions to show
\begin{equation}\label{equ:t/r}
    t(G) \ge r(P) \bigl(e(G)-\turannumber{n}{2}\bigr)
\end{equation}
for any $n$-vertex graph $G$, without requiring $K_4$-freeness.
Combined with \eqref{equ:t_e}, this immediately implies Theorem~\ref{Thm:original}.
To explain the proof of \eqref{equ:t/r} in more detail,
for any greedy partition $P$ in an $n$-vertex graph $G$ with $e$ edges, we define $r:=r(P)$, $t:=t(G)$, and $$g(G,P):=r(e-r(n-r))-t.$$
Using symmetrization arguments,
they \cite{GK2017} showed that it suffices to verify \eqref{equ:t/r} for complete multi-partite graphs $G$. 
Since $g(G,P)\le 0$ for any complete multipartite graph $G$ (see \cite[Lemma~8]{GK2017}), it follows that
\( \frac{t}{r} \ge e - r(n-r) \ge e - t_{n,2} \)
for such graphs, which establishes \eqref{equ:t/r} in full.

Motivated by this approach, Gy\H{o}ri and Keszegh~\cite{GK2017} proposed the following stronger conjecture.

\begin{conj}[Gy\H{o}ri and Keszegh~\cite{GK2017}, Conjecture 2]\label{Conj1.2}
Let $G$ be an $n$-vertex $K_4$-free graph with $e$ edges, and let $P$ be any greedy partition of $G$ of size $r := r(P)$.  
Then
\[
t(G) \ge r \bigl(e - r(n-r)\bigr), \quad \text{or equivalently, } \quad g(G,P) \le 0,
\]
and consequently, \(t_e(G) \ge e - r(n-r).\)
\end{conj}

In this paper, we first disprove Conjecture~\ref{Conj1.2} by constructing infinitely many counterexamples in a strong sense:
for some $K_4$-free graphs $G$ and greedy partitions $P$, the quantity $g(G,P)$ is postive and can be arbitrarily large.

\begin{thm}\label{Thm: Counter egs}
For any positive integer $\lambda$, there exists an $n$-vertex $K_4$-free graph $G$ with $e$ edges and a greedy partition $P$ of size $r$ such that
\(t(G) \le r \bigl(e - r(n-r)\bigr) - \lambda.\)
\end{thm}

Our proof of Theorem~\ref{Thm: Counter egs} begins by constructing four special graphs $F_1$, $F_2$, $F_3$, and $F_4$ (all defined in Section~\ref{sec:counterexamples}, each formed from at most three cliques).  
The first two are minimal counterexamples to Conjecture~\ref{Conj1.2}, while the ``$3$-blow-up" of $F_3$ and $F_4$ yield additional counterexamples.  
We then perform certain operations on these graphs to generate infinitely many larger, non-isomorphic counterexamples.

Our second contribution is to establish a corrected lower bound on the number of triangles in $K_4$-free graphs, using a new approach that is distinct from the method of Gy\H{o}ri and Keszegh~\cite{GK2017}.  
To state the result, we first introduce some notation. Let $P = \{T_1, T_2, \dots, T_r\}$ be any greedy partition of $G$ of size $r$. For indices $1 \le i < j < k \le r$, we say that a triple $(i,j,k)$ is \emph{$P$-bad} if the induced subgraph $G[T_i \cup T_j \cup T_k]$ is isomorphic to one of $F_1, F_2, F_3$, or $F_4$.  
We then define \(\omega(P)\) to be the total number of $P$-bad triples.  
The following result shows that the lower bound on the number of triangles in Conjecture~\ref{Conj1.2} can be corrected for all $K_4$-free graphs by subtracting \(\omega(P)\).

\begin{thm}\label{Thm:Main Thm}
Let $G$ be an $n$-vertex $K_4$-free graph with $e$ edges, and let $P$ be any greedy partition of $G$ of size $r$.  
Then
\[
t(G) \ge r \bigl(e - r(n-r)\bigr) - \omega(P).
\]  
In particular, if $G$ contains none of the induced subgraphs $F_1$, $F_2$, $F_3$, or $F_4$, the conclusion of Conjecture~\ref{Conj1.2} holds.
\end{thm}

The rest of the paper is organized as follows.
We prove Theorem~\ref{Thm: Counter egs} in Section~\ref{sec:counterexamples}.
The proof of Theorem~\ref{Thm:Main Thm} will be presented in Section~\ref{Sec:proof of main thm}.
In Section~\ref{Sec:Concluding Remarks}, we give some concluding remarks.

\section{Proof of Theorem~\ref{Thm: Counter egs}: Counterexamples to Conjecture~\ref{Conj1.2}}\label{sec:counterexamples}

In this section, we prove Theorem~\ref{Thm: Counter egs} by constructing infinitely many counterexamples to Conjecture~\ref{Conj1.2}.  
The constructions are divided into two types, presented in the following two subsections.

\subsection{Counterexamples of Type I}\label{subsec:type1}

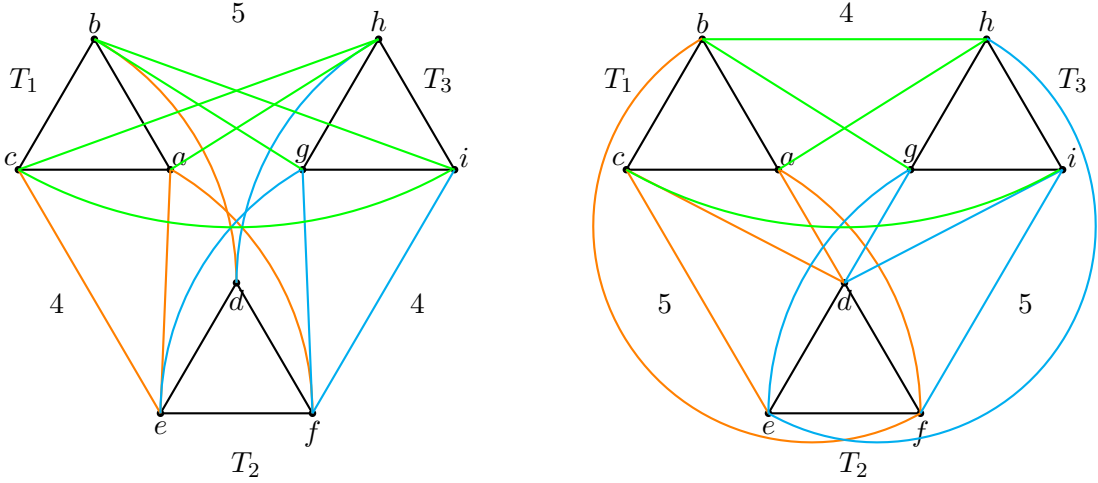
\begin{figure}[ht]\label{544and554}
\begin{center}
\begin{tikzpicture}
\hspace{-4cm}

      \tikzstyle{every node}=[draw,circle,fill=black,minimum size=2.0pt, inner sep=0pt]
      \draw [thick](-0.87,0.5) node (1) [label=70:$a$]{};
      \draw [thick](-1.87,2.23) node (2) [label=90:$b$]{};
      \draw [thick](-2.87,0.5) node (3) [label=110:$c$]{};
      \draw [thick](0,-1) node (4) [label=-90:$d$]{};
      \draw [thick](-1,-2.73) node (5) [label=-90:$e$]{};
      \draw [thick](1,-2.73) node (6) [label=-90:$f$]{};
      \draw [thick](0.87,0.5) node (7) [label=90:$g$]{};
      \draw [thick](1.87,2.23) node (8) [label=90:$h$]{};
      \draw [thick](2.87,0.5) node (9) [label=70:$i$]{};
      
      \draw [thick,color=black] (-0.87,0.5) -- (-1.87,2.23); 
      \draw [thick,color=black] (-0.87,0.5) -- (-2.87,0.5); 
      \draw [thick,color=black] (-2.87,0.5) -- (-1.87,2.23); 
      
      \draw [thick,color=black] (0,-1) -- (-1,-2.73); 
      \draw [thick,color=black] (0,-1) -- (1,-2.73); 
      \draw [thick,color=black] (-1,-2.73) -- (1,-2.73); 
     
      \draw [thick,color=black] (0.87,0.5) -- (1.87,2.23); 
      \draw [thick,color=black] (0.87,0.5) -- (2.87,0.5); 
      \draw [thick,color=black] (1.87,2.23) -- (2.87,0.5); 
   
      \draw [thick,color=orange] (-0.87,0.5) -- (-1,-2.73); 
      \draw [draw,thick,color=orange](1,-2.73) arc (0:60:3.74cm); 
      \draw [draw,thick,color=orange](0,-1) arc (0:60:3.73cm); 
      \draw [thick,color=orange] (-2.87,0.5) -- (-1,-2.73); 
      
      \draw [thick,color=cyan] (1,-2.73) -- (0.87,0.5);
      \draw [draw,thick,color=cyan](-1,-2.73) arc (180:120:3.74cm);
      \draw [draw,thick,color=cyan](0,-1) arc (180:120:3.73cm); 
      \draw [thick,color=cyan] (1,-2.73) -- (2.87,0.5); 
    
      \draw [thick,color=green] (-0.87,0.5) -- (1.87,2.23); 
      \draw [thick,color=green] (-1.87,2.23) -- (0.87,0.5); 
      \draw [thick,color=green] (-1.87,2.23) -- (2.87,0.5); 
      \draw [thick,color=green] (-2.87,0.5) -- (1.87,2.23); 
      \draw [draw,thick,color=green](-2.87,0.5) arc (-120:-60:5.7cm); 

      \put(-2,70){$5$}
      \put(-70,-40){$4$}
      \put(65,-40){$4$}

      \put(-85,45){$T_1$}
      \put(-2,-100){$T_2$}
      \put(70,45){$T_3$}
\hspace{8cm}

      \tikzstyle{every node}=[draw,circle,fill=black,minimum size=2.0pt, inner sep=0pt]
      \draw [thick](-0.87,0.5) node (1) [label=70:$a$]{};
      \draw [thick](-1.87,2.23) node (2) [label=90:$b$]{};
      \draw [thick](-2.87,0.5) node (3) [label=110:$c$]{};
      \draw [thick](0,-1) node (4) [label=-90:$d$]{};
      \draw [thick](-1,-2.73) node (5) [label=-90:$e$]{};
      \draw [thick](1,-2.73) node (6) [label=-90:$f$]{};
      \draw [thick](0.87,0.5) node (7) [label=90:$g$]{};
      \draw [thick](1.87,2.23) node (8) [label=90:$h$]{};
      \draw [thick](2.87,0.5) node (9) [label=70:$i$]{};
      
      \draw [thick,color=black] (-0.87,0.5) -- (-1.87,2.23); 
      \draw [thick,color=black] (-0.87,0.5) -- (-2.87,0.5);
      \draw [thick,color=black] (-2.87,0.5) -- (-1.87,2.23); 
      
      \draw [thick,color=black] (0,-1) -- (-1,-2.73); 
      \draw [thick,color=black] (0,-1) -- (1,-2.73); 
      \draw [thick,color=black] (-1,-2.73) -- (1,-2.73); 
     
      \draw [thick,color=black] (0.87,0.5) -- (1.87,2.23);
      \draw [thick,color=black] (0.87,0.5) -- (2.87,0.5);
      \draw [thick,color=black] (1.87,2.23) -- (2.87,0.5); 
      
      \draw [thick,color=orange] (-0.87,0.5) -- (0,-1); 
      \draw [thick,color=orange] (-2.87,0.5) -- (0,-1); 
      \draw [thick,color=orange] (-2.87,0.5) -- (-1,-2.73);
      \draw [draw,thick,color=orange](1,-2.73) arc (0:60:3.74cm); 
      \draw [draw,thick,color=orange](1,-2.73) arc (-60:-240:2.87cm); 
      
      \draw [thick,color=cyan] (0,-1) -- (0.87,0.5); 
      \draw [thick,color=cyan] (0,-1) -- (2.87,0.5); 
      \draw [thick,color=cyan] (1,-2.73) -- (2.87,0.5);
      \draw [draw,thick,color=cyan](-1,-2.73) arc (180:120:3.74cm); 
      \draw [draw,thick,color=cyan](-1,-2.73) arc (-120:60:2.87cm);
    
      \draw [thick,color=green] (-0.87,0.5) -- (1.87,2.23); 
      \draw [thick,color=green] (-1.87,2.23) -- (0.87,0.5); 
      \draw [thick,color=green] (-1.87,2.23) -- (1.87,2.23); 
      \draw [draw,thick,color=green](-2.87,0.5) arc (-120:-60:5.7cm);

      \put(-2,70){$4$}
      \put(-70,-40){$5$}
      \put(65,-40){$5$}

      \put(-90,45){$T_1$}
      \put(-2,-100){$T_2$}
      \put(80,45){$T_3$}
\end{tikzpicture}
\vspace{0.3cm}
\caption{Counterexamples $F_1$ (left) and $F_2$ (right) to Conjecture~\ref{Conj1.2}.}
\end{center}
\end{figure}

We define two graphs $F_1$ and $F_2$ as follows. 
Let $T_1=\{a,b,c\}$, $T_2=\{d,e,f\}$, and $T_3=\{g,h,i\}$ be three disjoint triangles.
\begin{itemize}
    \item [-] \textbf{The graph $F_1$}: Let $V(F_1)=T_1\cup T_2\cup T_3$ and $E(F_1)$ consists of 4 edges between $T_1$ and $T_2$, 4 edges between $T_2$ and $T_3$, and 5 edges between $T_1$ and $T_3$ (see Figure 1, on left). 
\end{itemize}

\noindent We see $P_1 = \{T_1,T_2,T_3\}$ defines a greedy partition of $F_1$. 
It is also easy to see that
$v(F_1) = 9, e(F_1) = 22, r(P_1) = 3$.
Thus this yields 
$$12=r(P_1)(e(F_1)-r(P_1)(v(F_1)-r(P_1))) >t(F_1) = 11.\footnote{There are 11 triangles in $F_1$: abc, ace, ach, aef, bci, bgi, chi, def, efg, fgi, ghi.}$$ 
Since any 4 vertices induce at most $2$ triangles, $F_1$ is $K_4$-free.
\begin{itemize}
    \item [-] \textbf{The graph $F_2$}: Let $V(F_2)=T_1\cup T_2\cup T_3$ and $E(F_2)$ consists of 5 edges between $T_1$ and $T_2$, 5 edges between $T_2$ and $T_3$, and 4 edges between $T_1$ and $T_3$ (see Figure 1, on right). 
\end{itemize}    

\noindent We see $P_2 = \{T_1,T_2,T_3\}$ defines a greedy partition of $F_2$. 
It is also easy to see that
$v(F_2)= 9, e(F_2) = 23,  r(P_2) = 3$.
Thus this yields        
$$15= r(P_2)(e(F_2)-r(P_2)(v(F_2)-r(P_2))) > t(F_2) =14.\footnote{There are 14 triangles in $F_2$:  abc, abf, abh, acd, adf, bgh, cde, cdi, def, deg, dfi, dgi, egh, ghi.}$$
Since any 4 vertices induce at most $2$ triangles, $F_2$ is also $K_4$-free.

\medskip

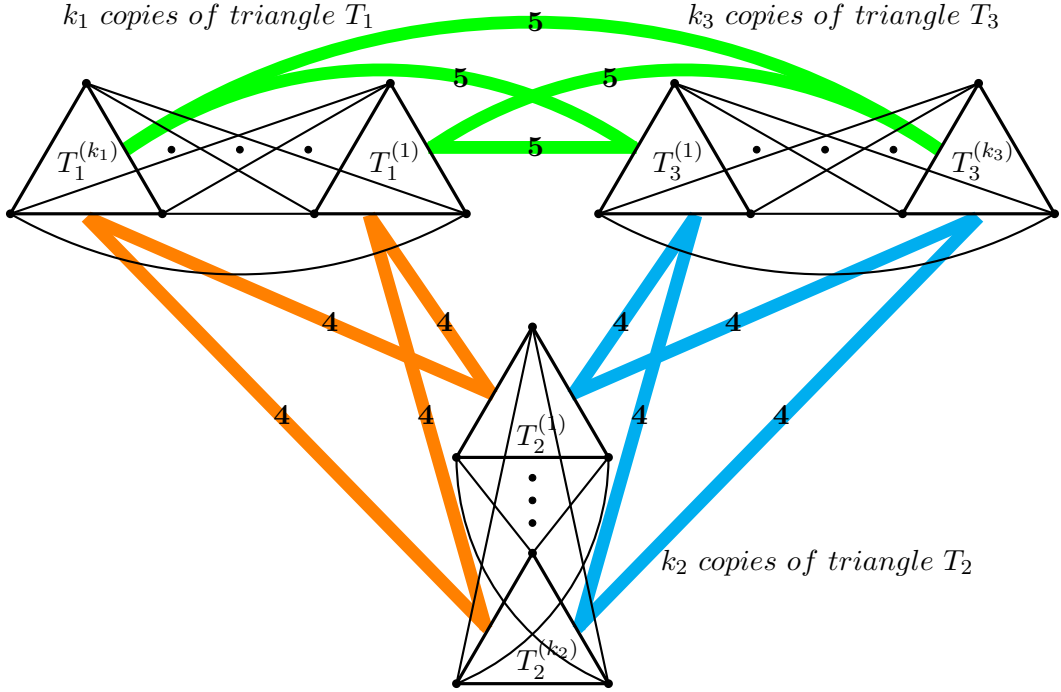
\begin{figure}[ht]\label{k1k2k3}
\begin{center}
\begin{tikzpicture}
 \tikzstyle{every node}=[draw,circle,fill=black,minimum size=2.5pt, inner sep=0pt]
      \draw [thick](-0.87,0.5) node (1) {};
      \draw [thick](-1.87,2.23) node (2) {};
      \draw [thick](-2.87,0.5) node (3) {};
      \draw [thick](0,-1) node (4) {};
      \draw [thick](-1,-2.73) node (5) {};
      \draw [thick](1,-2.73) node (6) {};
      \draw [thick](0.87,0.5) node (7) {};
      \draw [thick](1.87,2.23) node (8) {};
      \draw [thick](2.87,0.5) node (9) {};

      \draw [thick](-4.87,0.5) node (1') {};
      \draw [thick](-5.87,2.23) node (2') {};
      \draw [thick](-6.87,0.5) node (3') {};
      \draw [thick](0,-4) node (4') {}; 
      \draw [thick](-1,-5.73) node (5') {};
      \draw [thick](1,-5.73) node (6') {};
      \draw [thick](4.87,0.5) node (7') {};
      \draw [thick](5.87,2.23) node (8') {};
      \draw [thick](6.87,0.5) node (9') {};
      
      \draw (-2.95,1.35) node (a) {};
      \draw (-3.85,1.35) node (a) {};
      \draw (-4.75,1.35) node (a) {};
      \draw (2.95,1.35) node (b) {};
      \draw (3.85,1.35) node (b) {};
      \draw (4.75,1.35) node (b) {};
      \draw (0,-3) node (c) {};
      \draw (0,-3.3) node (c) {};
      \draw (0,-3.6) node (c) {};

      \draw [line width=1.7mm,color=orange] (-5.87,0.455) -- (-0.585,-5);
      \draw [line width=1.7mm,color=orange] (-5.85,0.45) -- (-0.5,-1.9);
      \draw [line width=1.7mm,color=orange] (-2.15,0.48) -- (-0.53,-1.9);
      \draw [line width=1.7mm,color=orange] (-2.15,0.48) -- (-0.61,-5);

      \draw [line width=1.7mm,color=cyan] (5.87,0.455) -- (0.585,-5);
      \draw [line width=1.7mm,color=cyan] (5.85,0.45) -- (0.5,-1.9);
      \draw [line width=1.7mm,color=cyan] (2.15,0.48) -- (0.53,-1.9);
      \draw [line width=1.7mm,color=cyan] (2.15,0.48) -- (0.61,-5);                  

      \draw [line width=1.7mm,color=green] (-1.37,1.38) -- (1.37,1.38);
      \draw [line width=1.7mm,color=green] (-1.37,1.35) arc (125:55:5.85cm);
      \draw [line width=1.7mm,color=green] (1.37,1.35) arc (55:125:5.85cm);
      \draw [line width=1.7mm,color=green] (-5.37,1.35) arc (125:55:9.35cm);

      \draw [very thick,color=black] (-0.87,0.5) -- (-1.87,2.23); 
      \draw [very thick,color=black] (-0.87,0.5) -- (-2.87,0.5); 
      \draw [very thick,color=black] (-2.87,0.5) -- (-1.87,2.23); 

      \draw [very thick,color=black] (0,-1) -- (-1,-2.73);
      \draw [very thick,color=black] (0,-1) -- (1,-2.73);
      \draw [very thick,color=black] (-1,-2.73) -- (1,-2.73); 

      \draw [very thick,color=black] (0.87,0.5) -- (1.87,2.23); 
      \draw [very thick,color=black] (0.87,0.5) -- (2.87,0.5); 
      \draw [very thick,color=black] (1.87,2.23) -- (2.87,0.5); 

      \draw [very thick,color=black] (-4.87,0.5) -- (-5.87,2.23);
      \draw [very thick,color=black] (-4.87,0.5) -- (-6.87,0.5); 
      \draw [very thick,color=black] (-6.87,0.5) -- (-5.87,2.23);

      \draw [very thick,color=black] (0,-4) -- (-1,-5.73);
      \draw [very thick,color=black] (0,-4) -- (1,-5.73);
      \draw [very thick,color=black] (-1,-5.73) -- (1,-5.73); 

      \draw [very thick,color=black] (4.87,0.5) -- (5.87,2.23); 
      \draw [very thick,color=black] (4.87,0.5) -- (6.87,0.5);
      \draw [very thick,color=black] (5.87,2.23) -- (6.87,0.5); 
      
      \draw [thick,color=black] (-0.87,0.5) -- (-5.87,2.23); 
      \draw [draw,thick,color=black](-0.87,0.5) arc (-60:-120:6cm);
      \draw [thick,color=black] (-1.87,2.23) -- (-4.87,0.5);
      \draw [thick,color=black] (-1.87,2.23) -- (-6.87,0.5);
      \draw [thick,color=black] (-2.87,0.5) -- (-4.87,0.5);
      \draw [thick,color=black] (-2.87,0.5) -- (-5.87,2.23);

      \draw [thick,color=black] (0,-1) -- (-1,-5.73); 
      \draw [thick,color=black] (0,-1) -- (1,-5.73);
      \draw [thick,color=black] (-1,-2.73) -- (0,-4);
      \draw [draw,thick,color=black](-1,-2.73) arc (180:248:3.25cm);
      \draw [thick,color=black] (1,-2.73) -- (0,-4);
      \draw [draw,thick,color=black](1,-2.73) arc (0:-68:3.25cm);
 
      \draw [thick,color=black] (0.87,0.5) -- (5.87,2.23); 
      \draw [draw,thick,color=black](6.87,0.5) arc (-60:-120:6cm);
      \draw [thick,color=black] (1.87,2.23) -- (4.87,0.5);
      \draw [thick,color=black] (1.87,2.23) -- (6.87,0.5);
      \draw [thick,color=black] (2.87,0.5) -- (4.87,0.5);
      \draw [thick,color=black] (2.87,0.5) -- (5.87,2.23);

      \put(-176,86){$k_1\,\,copies\,\,of\,\, triangle\,\, T_1$}
      \put(48,-120){$k_2\,\,copies\,\,of\,\, triangle\,\, T_2$}
      \put(58,86){$k_3\,\,copies\,\,of\,\, triangle\,\, T_3$}

      \put(-2,83){$\bf 5$}
      \put(-2,36){$\bf 5$}
      \put(-30,62){$\bf 5$}
      \put(26,62){$\bf 5$}
      
      \put(-36,-30){$\bf 4$}
      \put(-79,-30){$\bf 4$}
      \put(-97,-65){$\bf 4$}
      \put(-43,-65){$\bf 4$}

      \put(30,-30){$\bf 4$}
      \put(72,-30){$\bf 4$}
      \put(37,-65){$\bf 4$}
      \put(90,-65){$\bf 4$}

      \put(-61,30){$T_1^{(1)}$}
      \put(-6,-73){$T_2^{(1)}$}
      \put(45,30){$T_3^{(1)}$}

      \put(-178,30){$T_1^{(k_1)}$}
      \put(-6,-158){$T_2^{(k_2)}$}
      \put(157,30){$T_3^{(k_3)}$}
\end{tikzpicture}
\vspace{0.3cm}
\caption{Graph $F_1^{\mathbf{k}}$ with the greedy partition $P_1^{\mathbf{k}}$.}
\end{center}
\end{figure}

\begin{proof}[Proof of Theorem~\ref{Thm: Counter egs} (Type I)]
We will construct infinitely many counterexamples by blowing up $F_1$ and $F_2$.
For a positive integer vector $\mathbf{k}=(k_1,k_2,k_3)\in \mathbb{Z}_+^3$, for $i\in\{1,2\}$ and $j\in\{1,2,3\},$
the $\mathbf{k}$-blow-up of $F_i$, denoted by $F_i^{\mathbf{k}}$, is the graph
obtained by replacing every vertex $v$ of $T_j$ with $k_j$ different vertices where a copy of $u$ is adjacent to a copy of $v$ in $F_i^{\mathbf{k}}$ if and only if $u$ is adjacent to $v$ in $F_i$ (see Figure 2 for $F_1^{\mathbf{k}}$).
Note that the blow-up graph $F_i^{\mathbf{k}}$ is also  $K_4$-free.
We denote the $k_j$ copies of $T_j$ by $T_j^{(1)}$, $\dots$, $T_j^{(k_j)}$ for $j\in\{1,2,3\}$.
We see that $P_i^{\mathbf{k}} =\{T_1^{(1)},  \dots, T_1^{(k_1)}, T_2^{(1)}, \dots, T_2^{(k_2)}, T_3^{(1)}, \dots, T_3^{(k_3)}\}$ is a greedy partition of $F_i^{\mathbf{k}}$ for $i\in\{1,2\}.$
It is not hard to see that
$v(F_1^{\mathbf{k}})= 3(k_1 + k_2 + k_3), 
e(F_1^{\mathbf{k}}) = 3(k_1^2 + k_2^2 + k_3^2) + 4k_1k_2 + 5k_1k_3 + 4k_2k_3, 
r(P_1^{\mathbf{k}}) = k_1 + k_2 + k_3$ and 
$t(F_1^{\mathbf{k}})= k_1^3 + k_2^3 + k_3^3 + k_1^2k_2 + 2k_1^2k_3 + k_2^2k_1 + k_2^2k_3  + 2k_3^2k_1 + k_3^2k_2.$
Thus this yields    
$$t(F_1^{\mathbf{k}}) - r(P_1^{\mathbf{k}})(e(F_1^{\mathbf{k}}) - r(P_1^{\mathbf{k}})(v(F_1^{\mathbf{k}}) - r(P_1^{\mathbf{k}}))) = -k_1k_2k_3 < 0.$$

Similarly, it is also not hard to see that
$v(F_2^{\mathbf{k}})= 3(k_1 + k_2 + k_3), 
e(F_2^{\mathbf{k}}) = 3(k_1^2 + k_2^2 + k_3^2) + 5k_1k_2 + 4k_1k_3 + 5k_2k_3,  
r(P_2^{\mathbf{k}}) = k_1 + k_2 + k_3$ and 
$t(F_2^{\mathbf{k}})= k_1^3 + k_2^3 + k_3^3 + 2k_1^2k_2 + k_1^2k_3 + 2k_2^2k_1 + 2k_2^2k_3 + k_3^2k_1 + 2k_3^2k_2 + k_1k_2k_3.$
Thus this yields    
$$t(F_2^{\mathbf{k}}) - r(P_2^{\mathbf{k}})(e(F_2^{\mathbf{k}}) - r(P_2^{\mathbf{k}})(v(F_2^{\mathbf{k}}) - r(P_2^{\mathbf{k}}))) = -k_1k_2k_3 < 0.$$
Therefore the graphs $F_1^{\mathbf{k}}$ with greedy partition $P_1^{\mathbf{k}}$ and $F_2^{\mathbf{k}}$ with greedy partition $P_2^{\mathbf{k}}$ are counterexamples to Conjecture~\ref{Conj1.2}.
Moreover, as $n=3(k_1+k_2+k_3),$ the discrepancy $r(e - r(n - r))-t=k_1k_2k_3$ approaches infinity as $n\to \infty$, which completes the proof of Theorem~\ref{Thm: Counter egs}.
\end{proof}

\subsection{Counterexamples of Type II}\label{subsec:type2}
We define two graphs $F_3$ and $F_4$ as follows.
Let $T_1=\{a,b,c\}$, $T_2=\{d,e,f\}$, and $T_3=\{g,h,i\}$ be three disjoint triangles.
\begin{itemize}
    \item [-] \textbf{The graph $F_3$}: Let $V(F_3)=T_1\cup T_2\cup T_3$ and $E(F_3)$ consists of 5 edges between $T_1$ and $T_2$, 5 edges between $T_2$ and $T_3$, and 3 edges between $T_1$ and $T_3$ (see Figure 3, on left).    
\end{itemize}
It is easy to see that $v(F_3)=9,$ $e(F_3)=22$ and $t(F_3)=13.\footnote{There are 13 triangles in $F_3$: abc, abf, abh, acd, adf, bgh, cde, def, deg, dfi, dgi, egh, ghi.}$ 
Since any 4 vertices induce at most 2 triangles, $F_3$ is $K_4$-free.
\medskip

Let $T_1'=\{a,b,c\}$ and $T_3'=\{f,g,h\}$ be two disjoint triangles and $T_2'=\{d,e\}$ be an edge.
\begin{itemize}
    \item [-] \textbf{The graph $F_4$}: Let $V(F_4)=T'_1\cup T'_2\cup T'_3$ and 
    $E(F_4)$ consists of 3 edges between $T'_1$ and $T'_2$, 3 edges between $T'_2$ and $T'_3$, and 5 edges between $T'_1$ and $T'_3$ (see Figure 3, on right).
\end{itemize}
It is easy to see that $v(F_4)=8,$ $e(F_4)=18$ and $t(F_4)=10.\footnote{There are 10 triangles in $F_4$: abc,  acd, acg,  ade, bch, bfh, cgh, def, efh, fgh.}$ 
Since any 4 vertices induce at most 2 triangles, $F_4$ is also $K_4$-free.

\vspace{0.5cm}
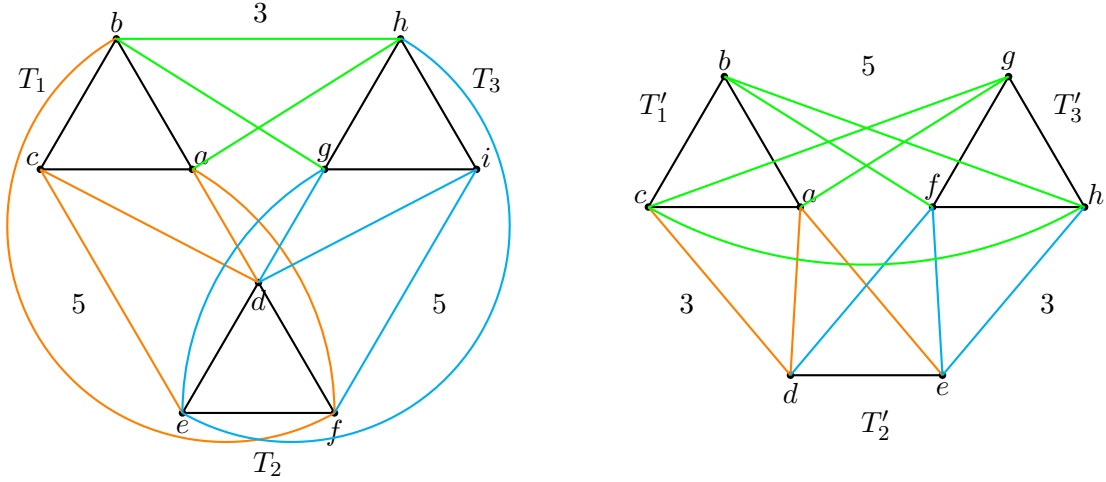
\begin{figure}[ht]\label{553}
\begin{center}
\begin{tikzpicture}
\hspace{-4cm}
\vspace{0.1cm}
      \tikzstyle{every node}=[draw,circle,fill=black,minimum size=2.0pt, inner sep=0pt]
      \draw [thick](-0.87,0.5) node (1) [label=70:$a$]{};
      \draw [thick](-1.87,2.23) node (2) [label=90:$b$]{};
      \draw [thick](-2.87,0.5) node (3) [label=110:$c$]{};
      \draw [thick](0,-1) node (4) [label=-90:$d$]{};
      \draw [thick](-1,-2.73) node (5) [label=-90:$e$]{};
      \draw [thick](1,-2.73) node (6) [label=-90:$f$]{};
      \draw [thick](0.87,0.5) node (7) [label=90:$g$]{};
      \draw [thick](1.87,2.23) node (8) [label=90:$h$]{};
      \draw [thick](2.87,0.5) node (9) [label=70:$i$]{};
      
      \draw [thick,color=black] (-0.87,0.5) -- (-1.87,2.23); 
      \draw [thick,color=black] (-0.87,0.5) -- (-2.87,0.5);
      \draw [thick,color=black] (-2.87,0.5) -- (-1.87,2.23); 
    
      \draw [thick,color=black] (0,-1) -- (-1,-2.73);
      \draw [thick,color=black] (0,-1) -- (1,-2.73); 
      \draw [thick,color=black] (-1,-2.73) -- (1,-2.73); 
    
      \draw [thick,color=black] (0.87,0.5) -- (1.87,2.23); 
      \draw [thick,color=black] (0.87,0.5) -- (2.87,0.5);
      \draw [thick,color=black] (1.87,2.23) -- (2.87,0.5); 

      \draw [thick,color=orange] (-0.87,0.5) -- (0,-1); 
      \draw [thick,color=orange] (-2.87,0.5) -- (0,-1); 
      \draw [thick,color=orange] (-2.87,0.5) -- (-1,-2.73);
      \draw [draw,thick,color=orange](1,-2.73) arc (0:60:3.74cm); 
      \draw [draw,thick,color=orange](1,-2.73) arc (-60:-240:2.87cm); 

      \draw [thick,color=cyan] (0,-1) -- (0.87,0.5); 
      \draw [thick,color=cyan] (0,-1) -- (2.87,0.5);
      \draw [thick,color=cyan] (1,-2.73) -- (2.87,0.5); 
      \draw [draw,thick,color=cyan](-1,-2.73) arc (180:120:3.74cm); 
      \draw [draw,thick,color=cyan](-1,-2.73) arc (-120:60:2.87cm); 
    
      \draw [thick,color=green] (-0.87,0.5) -- (1.87,2.23); 
      \draw [thick,color=green] (-1.87,2.23) -- (0.87,0.5); 
      \draw [thick,color=green] (-1.87,2.23) -- (1.87,2.23); 

      \put(-2,70){$3$}
      \put(-70,-40){$5$}
      \put(65,-40){$5$}

      \put(-90,45){$T_1$}
      \put(-2,-100){$T_2$}
      \put(80,45){$T_3$}
\hspace{8cm}

      \tikzstyle{every node}=[draw,circle,fill=black,minimum size=2.0pt, inner sep=0pt]
      \draw [thick](-0.87,0) node (1) [label=70:$a$]{};
      \draw [thick](-1.87,1.73) node (2) [label=90:$b$]{};
      \draw [thick](-2.87,0) node (3) [label=110:$c$]{};
      
      \draw [thick](-1,-2.23) node (5) [label=-90:$d$]{};
      \draw [thick](1,-2.23) node (6) [label=-90:$e$]{};
      \draw [thick](0.87,0) node (7) [label=90:$f$]{};
      \draw [thick](1.87,1.73) node (8) [label=90:$g$]{};
      \draw [thick](2.87,0) node (9) [label=70:$h$]{};
      
      \draw [thick,color=black] (-0.87,0) -- (-1.87,1.73); 
      \draw [thick,color=black] (-0.87,0) -- (-2.87,0);
      \draw [thick,color=black] (-2.87,0) -- (-1.87,1.73); 
      
      \draw [thick,color=black] (-1,-2.23) -- (1,-2.23); 
  
      \draw [thick,color=black] (0.87,0) -- (1.87,1.73); 
      \draw [thick,color=black] (0.87,0) -- (2.87,0); 
      \draw [thick,color=black] (1.87,1.73) -- (2.87,0); 

      \draw [thick,color=orange] (-0.87,0) -- (-1,-2.23); 
      \draw [thick,color=orange] (-0.87,0) -- (1,-2.23);
   
      \draw [thick,color=orange] (-2.87,0) -- (-1,-2.23); 
    
      \draw [thick,color=cyan] (1,-2.23) -- (0.87,0); 
      \draw [thick,color=cyan] (-1,-2.23) -- (0.87,0); 
    
      \draw [thick,color=cyan] (1,-2.23) -- (2.87,0); 

      \draw [thick,color=green] (-0.87,0) -- (1.87,1.73); 
      \draw [thick,color=green] (-1.87,1.73) -- (0.87,0); 
      \draw [thick,color=green] (-1.87,1.73) -- (2.87,0); 
      \draw [thick,color=green] (-2.87,0) -- (1.87,1.73);
      \draw [draw,thick,color=green](-2.87,0) arc (-120:-60:5.7cm); 

      \put(-2,50){$5$}
      \put(-70,-40){$3$}
      \put(65,-40){$3$}

      \put(-85,35){$T'_1$}
      \put(-2,-85){$T'_2$}
      \put(70,35){$T'_3$}

\end{tikzpicture}
\vspace{0.3cm}
\caption{Graphs $F_3$ (left) and $F_4$ (right).}
\end{center}
\end{figure}
\vspace{-0.5cm}

\begin{proof}[Proof of Theorem~\ref{Thm: Counter egs} (Type II)]
While $F_3$ and $F_4$ are not counterexamples to Conjecture~\ref{Conj1.2}, some blow-ups of these graphs are.
Indeed, by a similar proof of Theorem~\ref{Thm: Counter egs} (Type I),
for a positive integer vector $\mathbf{k}=(k_1,k_2,k_3)\in \mathbb{Z}_+^3$, for $j\in\{1,2,3\},$
we define $F_3^{\mathbf{k}}$ (respectively, $F_4^{\mathbf{k}}$) to be the graph
obtained by replacing every vertex $v$ of $T_j$ (respectively, $T'_j$) with $k_j$ different vertices.
We denote the $k_j$ copies of $T_j$ by $T_j^{(1)}$, $\dots$, $T_j^{(k_j)}$ for $j\in\{1,2,3\}$.
We see that $P_3^{\mathbf{k}} =\{T_1^{(1)},  \dots, T_1^{(k_1)}, T_2^{(1)}, \dots, T_2^{(k_2)}, T_3^{(1)}, \dots, T_3^{(k_3)}\}$ is a greedy partition of $F_3^{\mathbf{k}}$.
It is not hard to see that 
$v(F_3^{\mathbf{k}})= 3(k_1+k_2+k_3),$
$e(F_3^{\mathbf{k}})= 3(k_1^2 + k_2^2 + k_3^2) + 5k_1k_2 + 3k_1k_3 + 5k_2k_3,$
$r(P_3^{\mathbf{k}})= k_1+k_2+k_3,$ and
$t(F_3^{\mathbf{k}})= k_1^3 + k_2^3 + k_3^3 + 2k_1^2k_2 + k_1^2k_3 + 2k_2^2k_1 + 2k_2^2k_3  + k_3^2k_1 + 2k_3^2k_2.$
Thus this yields
$$t(F_3^{\mathbf{k}}) - r(P_3^{\mathbf{k}})(e(F_3^{\mathbf{k}}) - r(P_3^{\mathbf{k}})(v(F_3^{\mathbf{k}}) - r(P_3^{\mathbf{k}}))) = k_1k_3(k_1 - k_2 + k_3).$$
Thus for infinitely many $(k_1,k_2,k_3)$ satisfying the inequality $k_1 + k_3 < k_2$, 
the graph $F_3^{\mathbf{k}}$ with greedy partition $P_3^{\mathbf{k}}$ is a counterexample to Conjecture~\ref{Conj1.2}. 
Moreover, as $n=3(k_1+k_2+k_3)$, the discrepancy $r(e - r(n - r))-t= k_1k_3(k_2 -k_1 - k_3)$ may approach infinity as $n\to \infty$. 

Similarly, we denote the $k_j$ copies of $T'_j$ by $T_j^{'(1)}$, $\dots$, $T_j^{'(k_j)}$ for $j\in\{1,2,3\}$.
We see that $P_4^{\mathbf{k}} =\{T_1^{'(1)},  \dots, T_1^{'(k_1)}, T_2^{'(1)}, \dots, T_2^{'(k_2)}, T_3^{'(1)}, \dots, T_3^{'(k_3)}\}$ defines a greedy partition of $F_4^{\mathbf{k}}$.
It is not hard to see that 
$v(F_4^{\mathbf{k}})= 3k_1 + 2k_2 + 3k_3,$
$e(F_4^{\mathbf{k}})= 3k_1^2 + k_2^2 + 3k_3^2 + 3k_1k_2 + 5k_1k_3 + 3k_2k_3,$
$r(P_4^{\mathbf{k}})= k_1+k_2+k_3,$ and
$t(F_4^{\mathbf{k}})= k_1^3 + k_3^3 + k_1^2k_2 + 2k_1^2k_3 + k_2^2k_1 + k_2^2k_3 + 2k_3^2k_1 + k_3^2k_2.$
Thus this yields
$$t(F_4^{\mathbf{k}}) - r(P_4^{\mathbf{k}})(e(F_4^{\mathbf{k}}) - r(P_4^{\mathbf{k}})(v(F_4^{\mathbf{k}}) - r(P_4^{\mathbf{k}}))) = k_2(k_1k_2 - k_1k_3 + k_2k_3).$$
Thus for infinitely many $(k_1,k_2,k_3)$ satisfying the inequality $k_1k_2 + k_2k_3 < k_1k_3$, 
the graph $F_4^{\mathbf{k}}$ with greedy partition $P_4^{\mathbf{k}}$ is a counterexample to Conjecture~\ref{Conj1.2}.
Moreover, as $n=3k_1+2k_2+3k_3$, the discrepancy $r(e - r(n - r))-t= k_2(k_1k_3- k_1k_2 - k_2k_3)$  may approach infinity as $n\to \infty$, which completes the proof of Theorem~\ref{Thm: Counter egs}.
\end{proof}

\section{Proof of Theorem~\ref{Thm:Main Thm}}\label{Sec:proof of main thm}

In this section, we first reduce the proof of Theorem~\ref{Thm:Main Thm} to a key lemma (Lemma~\ref{Lem:KeyLem}) in Subsection~\ref{Sec:proof of Thm}, and then prove Lemma~\ref{Lem:KeyLem} in Subsection~\ref{Sec:proof of Key lemma}.  
A proof outline of Theorem~\ref{Thm:Main Thm} is also provided in Subsection~\ref{Sec:proof of Thm}.

Throughout the rest of this section,
let $G$ be an $n$-vertex $K_4$-free graph with $e$ edges, and
let $P=\{T_1,...,T_r\}$ be a greedy partition of $G$ of size $r$.
Since $G$ is $K_4$-free,
each $T_i$ has size $3,2$ or $1.$
Let $a,b,c$ be the number of cliques of size $3,2$ and $1$, respectively.
Thus $r=a+b+c,$ and $n=3a+2b+c.$
\begin{itemize}
\item For $1\leq i<j\leq r$, we define
$$e_{ij} = e(G[T_i \cup T_j])\quad \text{and}\quad  t_{ij}= \text{the number of triangles in}\  G[T_i \cup T_j].$$
\item For $1\leq i<j<k \leq r$, we define
$$e_{ijk} = e(G[T_i \cup T_j \cup T_k])\quad \text{and}\quad t_{ijk} = \text{the number of triangles in}\  G[T_i \cup T_j \cup T_k].$$
\end{itemize}

\subsection{Completing the Proof, Assuming Lemma~\ref{Lem:KeyLem}}\label{Sec:proof of Thm}

Our proof strategy of Theorem~\ref{Thm:Main Thm} employs double counting technique to analyze the contribution of triangles. 
Specifically, by double counting the number of triangles that contribute to $t_{ijk}$, 
we can write $t$ as a sum of these $t_{ijk}$ terms and a term related to $t_{ij}$ (see equation~\eqref{Equ:t(G)}).
Next, in Lemma~\ref{Lem:KeyLem}, we analyze how $t_{ij}$ affects $t_{ijk}$ locally, 
then apply the induction to extend the effect to the whole graph, 
and thus obtain the desired lower bound on the number of triangles.

For $i\in[3],$ let $M_i(G;P)$ be the number of triangles with three vertices lying in exactly $i$ different $T_j$'s.
Clearly, we have $M_1(G;P)=a,$ and $t(G)=M_1(G;P)+M_2(G;P)+M_3(G;P)$.
By double counting the number of triangles, we have
\begin{align*}
    \sum_{1\leq i<j<k \leq r}t_{ijk}&= M_1(G;P)\cdot \binom{r-1}{2}+M_2(G;P)\cdot \binom{r-2}{1}+M_3(G;P)\\
    &=t(G)+(r-3)M_2(G;P)+\frac{a}{2}(r-1)(r-2)-a,
\end{align*}
which implies that
\begin{align}\label{Equ:t(G)}
    t(G)=\sum_{1\leq i<j<k \leq r}t_{ijk}-(r-3)M_2(G;P)-\frac{a}{2}(r-1)(r-2)+a.
\end{align}

To lower bound the right-hand side of~\eqref{Equ:t(G)}, we first estimate $M_2(G;P)$. 
We define the pair deficiency
\begin{align}\label{Equ:defM0}
     M_0(G;P):=\sum_{1\leq i<j \leq r}2(e_{ij}-2(|T_i|+|T_j|-2))-a(r-1).
\end{align}

The following lemma states that $M_2(G;P)$ is bounded below by the pair deficiency $M_0(G;P).$

\begin{lem}\label{Lem:M2>M0}
    Let $G$ be a $K_4$-free graph with greedy partition $P$.
    Then we have $M_2(G;P)\ge M_0(G;P).$
\end{lem}
\begin{proof}
We first claim that for any $1\leq i<j \leq r$, we have $t_{ij}\ge 2(e_{ij}-2(|T_i|+|T_j|-2)).$
Indeed, this follows from case analysis of all possible sizes for $T_i$ and $T_j$ (which can only be 1, 2, or 3 vertices due to the $K_4$-free condition). 
We omit the detail.

Summing up all the $t_{ij}$'s and by~\eqref{Equ:defM0}, we obtain that
$$\sum_{1\leq i<j\leq r} t_{ij}\ge \sum_{1\leq i<j\leq r} 2\left(e_{ij}-2(|T_i|+|T_j|-2)\right)=
     M_0(G;P)+ a(r-1).$$
By the definition of the greedy partition, we know that $G[T_{a+1},...,T_r]$ is triangle-free.
By double counting the number of triangles in $G[T_i\cup T_j],$ we have
$$\sum_{1\leq i<j\leq r} t_{ij}=M_2(G;P)+ a(r-1),$$
and this implies that $M_2(G;P)\geq M_0(G;P),$
which completes the proof of Lemma~\ref{Lem:M2>M0}.
\end{proof}

Next, we give a lower bound on the term $\sum_{1\leq i<j<k \leq r}t_{ijk}$ in~\eqref{Equ:t(G)}.
We define the triple deficiency motivated by the bound suggested in Conjecture~\ref{Conj1.2}
\begin{align}\label{Equ:defF0}
F_0(G;P):=\sum_{1\leq i<j<k \leq r}3(e_{ijk}-3(|T_i|+|T_j|+|T_k|-3)).
\end{align}

We present our key lemma as follows and postpone the proof to Subsection~\ref{Sec:proof of Key lemma}.
    \begin{lem}\label{Lem:KeyLem}
        Let $G$ be a $K_4$-free graph and $P$ be any greedy partition of $G$ of size $r$.
        Let $M_2(G;P)=M_0(G;P)+C$ for some $C\ge 0$,
        then we have $\sum\limits_{1\leq i<j<k \leq r} t_{ijk}\geq F_0(G;P)+ (r-2)\cdot C-\omega (P)$.
    \end{lem}

Now we are ready to finish the proof of Theorem~\ref{Thm:Main Thm}.

\medskip
\noindent
\textbf{Proof of Theorem~\ref{Thm:Main Thm}.} 
By Lemma~\ref{Lem:M2>M0}, we may assume that $M_2(G;P)=M_0(G;P)+C$ for some $C \geq 0$.
By Lemma~\ref{Lem:KeyLem}, it follows that
\begin{align}\label{Equ:tijk}
\sum_{1\leq i<j<k \leq r} t_{ijk}\geq F_0(G;P)+ (r-2)\cdot C-\omega (P).
\end{align}
Combining \eqref{Equ:t(G)} with~\eqref{Equ:tijk}, we have
\begin{align}\label{Equ:t(G)>r(e-r(n-r))-w(P)}
    t(G)\geq& F_0(G;P)+(r-2)\cdot C-\omega(P)-(r-3)[M_0(G;P)+C]-\frac{a}{2}(r-1)(r-2)+a\nonumber\\
    =&F_0(G;P)-(r-3)M_0(G;P)-\frac{a}{2}(r-1)(r-2)+a-\omega(P)+C\nonumber\\
    =& r(e-r(n-r))-\omega(P)+C \geq r(e-r(n-r))-\omega(P),
\end{align}
where the last equality holds by the definitions of $M_0(G;P)$ and $F_0(G;P)$ (see its justification in Appendix~A).
This completes the proof of Theorem~\ref{Thm:Main Thm}.
\QED

\subsection{Proof of Lemma~\ref{Lem:KeyLem}}\label{Sec:proof of Key lemma}

In this subsection, we finish the proof of Lemma~\ref{Lem:KeyLem} by induction on $r,$ the size of the greedy partition $P$ of $G.$
First, we prove the base case for $r\leq3$ in the following claim.

\begin{claim}\label{Clm:BaseCase}
    Let $G$ be a $K_4$-free graph and $P$ be any greedy partition of $G$ of size $r$ for some $r\leq3$. 
    Let $M_2(G;P)=M_0(G;P)+C$ for some $C\ge 0$, then we have
    $$\sum_{1\leq i<j<k \leq r} t_{ijk}\geq F_0(G;P)+ (r-2)\cdot C-\omega (P).$$
\end{claim}
\begin{proof}
When $r=1$ or $2,$ by definition, we have
$\sum_{1\leq i<j<k \leq r} t_{ijk}=F_0(G;P)=\omega(P)=0,$ and we are done.
The verification of $r=3$  involves detailed case analysis and computer assistance.
We defer the proof to Appendix B.
\end{proof}

Now we are ready to finish the proof of Lemma~\ref{Lem:KeyLem}.

\medskip
\noindent
\textbf{Completing the proof of Lemma~\ref{Lem:KeyLem}.}
By Claim~\ref{Clm:BaseCase}, we may assume that the statement is true for any $K_4$-free graph with any greedy partition of size at most $r-1$ for some $r\ge4$.
We now consider a $K_4$-free graph $G$ with greedy partition $P=\{T_1,...,T_r\}$ of size $r$.

Let $a$ be the number of $T_i$'s in $P$ of size $3$.
For $\ell \in [r]$,
let $G_\ell=G[V(G)\setminus T_\ell]$,
$P_\ell=P\setminus\{T_\ell\}$ (a greedy partition of $G_\ell$ of size $r-1$),
and let
$$M_2(G_\ell,T_\ell):=M_2(G;P)-M_2(G_\ell;P_\ell).$$

By Lemma~\ref{Lem:M2>M0},
suppose that $M_2(G;P)=M_0(G;P)+C$ for some $C\ge 0$.
Then for $\ell \in [r]$,
\begin{align*}
    M_2(G_\ell;P_\ell)&=M_2(G;P)-M_2(G_\ell, T_\ell)
    =M_0(G;P)+C-M_2(G_\ell, T_\ell)\\[3pt]
    &=M_0(G_\ell;P_\ell)+[M_0(G;P)+C-M_2(G_\ell, T_\ell)-M_0(G_\ell;P_\ell)].
\end{align*}
Let $C_\ell:=M_0(G;P)+C-M_2(G_\ell, T_\ell)-M_0(G_\ell;P_\ell)$.
By Lemma~\ref{Lem:M2>M0} again, 
we have $C_\ell\geq 0 $ for $\ell\in[r]$.

Let $I_\ell$ be the set of all triples $(i,j,k)$ such that $1\le i<j<k\le r$ and
$i,j,k\in [r]\setminus\{\ell\}.$
By the inductive hypothesis,
we have
\begin{align}\label{Equ:induction when 1<ell<a}
    \sum_{(i,j,k)\in I_\ell}t_{ijk}
    \geq F_0(G_\ell;P_\ell)+ (r-3) \cdot C_\ell -\omega (P_\ell).
\end{align}
Summing up all $\ell\in[r]$, by~\eqref{Equ:induction when 1<ell<a} and the definition of $C_\ell$, we have
\begin{align}\label{Inequ:tijk}
    \sum_{\ell\in[r]}\,\,\sum_{(i,j,k)\in I_\ell}t_{ijk}
    \geq
    &\sum_{\ell\in [r]}F_0(G_\ell; P_\ell)
    +(r-3)\cdot\sum_{\ell\in [r]} C_\ell -\sum_{\ell\in [r]}\omega (P_\ell)\nonumber\\
    =& X+(r-3)\left[r(M_0(G;P)+C)-\sum_{\ell\in [r]} M_2(G_\ell, T_\ell)-Y\right]-\sum_{\ell\in [r]}\omega (P_\ell),
\end{align}
where $X=\sum_{\ell\in [r]}F_0(G_\ell; P_\ell),$
and $Y=\sum_{\ell\in [r]} M_0(G_\ell;P_\ell).$

For the left-hand side of~\eqref{Inequ:tijk}, by double counting the contribution of $t_{ijk},$  we have
\begin{align}\label{Equ:sum tijk}
    \sum_{\ell\in[r]}\,\,\sum_{(i,j,k)\in I_\ell}t_{ijk}=
    (r-3)\cdot\sum_{1\leq i<j<k \leq r}t_{ijk}.
\end{align}
Recall that $F_0(G;P)=\sum_{1\leq i<j<k \leq r}3(e_{ijk}-3(|T_i|+|T_j|+|T_k|-3))$ (see~\eqref{Equ:defF0}),
for any triple $1\leq i<j<k \leq r$,
the term $3(e_{ijk}-3(|T_i|+|T_j|+|T_k|-3))$ is counted $r-3$ times in $X$.
Consequently,
\begin{align}\label{Equ:X}
    X=(r-3)\cdot F_0(G;P).
\end{align}
Similarly, by the definition of $M_0(G;P)=\sum_{1\leq i<j \leq r}2(e_{ij}-2(|T_i|+|T_j|-2))-a(r-1)$ (see~\eqref{Equ:defM0}),
each term $2(e_{ij}-2(|T_i|+|T_j|-2))$ is counted $r-2$ times in $Y$.
Moreover, the term $a(r-2)$ appears $r-a$ times in $Y$ (for $\ell \in [r]\setminus [a]$), while the term $(a-1)(r-2)$ appears $a$ times in $Y$ (for $\ell \in [a]$).
Thus,
\begin{align}\label{Equ:Y}
    Y&=(r-2)\sum_{1\leq i<j \leq r}2(e_{ij}-2(|T_i|+|T_j|-2))-a(r-2)(r-a)-(a-1)(r-2)a\nonumber\\
     &=(r-2)\sum_{1\leq i<j \leq r}2(e_{ij}-2(|T_i|+|T_j|-2))-a(r-1)(r-2)=(r-2)\cdot M_0(G;P).
\end{align}
Finally, by double counting the number of triangles and induced subgraphs isomorphic to $F_1,$ $F_2,$ $F_3$ or $F_4$ contributing to $M_2(G;P)$ and $\omega(P)$, respectively, we have
\begin{align}\label{Equ:M2 and w}
    \sum_{\ell\in [r]}M_2(G_\ell, T_\ell)=2M_2(G;P), \ \text{and}\
    \sum_{\ell\in [r]}\omega (P_\ell)=(r-3)\cdot \omega(P).
\end{align}
Combining equations~\eqref{Equ:sum tijk},~\eqref{Equ:X},~\eqref{Equ:Y}, and~\eqref{Equ:M2 and w} with inequality~\eqref{Inequ:tijk}, we obtain that
\begin{align*}
    (r-3)\sum_{1\leq i<j<k \leq r}t_{ijk}
    \geq &(r-3)\cdot F_0(G;P)
    +(r-3)\cdot[r(M_0(G;P)+C)-2M_2(G;P)-\\
    &(r-2)\cdot M_0(G;P)] - (r-3)\cdot \omega(P).
\end{align*}
Since $M_2(G;P)=M_0(G;P)+C,$ we conclude that
\begin{align*}
    \sum_{1\leq i<j<k \leq r}t_{ijk}&\geq
    F_0(G;P)+2M_0(G;P)-2M_2(G;P)+r C-\omega(P)\\
    &=F_0(G;P)+(r-2)\cdot C-\omega(P),
\end{align*}
which completes the proof of Lemma~\ref{Lem:KeyLem}.
\QED

\section{Concluding Remarks}\label{Sec:Concluding Remarks}

In this paper, we first disprove the assertion in Conjecture~\ref{Conj1.2}, proposed by Gy\H{o}ri and Keszegh~\cite{GK2017}, concerning the number $t(G)$ of triangles in $K_4$-free graphs $G$ under size and greedy partition constraints.  
Second, we provide a corrected lower bound, showing that $t(G) \ge r(e - r(n-r)) - \omega(P)$ for any $n$-vertex $K_4$-free graph $G$ with $e$ edges and any greedy partition $P$ of size $r$.  
It would be interesting to further improve this bound, particularly the term involving $\omega(P)$.

We remark that the bound given by Theorem~\ref{Thm:Main Thm} can be tight for infinitely many $K_4$-free graphs. 
Let $G$ be any complete $3$-partite graph with parts $X,Y$ and $Z$. 
Suppose that the sizes of $X,Y$ and $Z$ are $x,y$ and $z$, respectively and $x\ge y\ge z.$ 
It is easy to see that the greedy partition $P$ of $G$ is unique, which consists of $z$ triangles, $y-z$ edges and $x-y$ isolated vertices.
Thus we have $n=x+y+z,$  $e=xy+xz+yz,$  $r=r(P)=x,$ and 
moreover it is easy to see that $\omega(P)=0.$
By Theorem~\ref{Thm:Main Thm}, we have $xyz=t\ge r(e-r(n-r))=xyz,$ thus our theorem is tight for complete $3$-partite graph.
For any positive integer vector $\mathbf{k}=(k_1,k_2,k_3)\in \mathbb{Z}_+^3$ and $i\in\{1, 2\},$ Theorem~\ref{Thm:Main Thm} is also tight for graphs $F_i^{\mathbf{k}}$ that we constructed in Subsection~\ref{subsec:type1}, since the value $\omega(P_i^{\mathbf{k}}) = k_1k_2k_3$ exactly matches the gap between $t$ and $r(e-r(n-r))$.

Our second remark concerns the maximum of the quantity $g(G,P)$ as a function of $n$.  
Formally, let $g(n) = \max_{(G,P)} g(G,P)$, where the maximum is taken over all $n$-vertex $K_4$-free graphs $G$ and all greedy partitions $P$ of $G$.  
We claim that $g(n) = \Theta(n^3)$.
Indeed, for any greedy partition $P$ of $G$ of size $r$,
$\omega(P)\leq \binom{r}{3}\le \binom{n}{3}$.
By Theorem~\ref{Thm:Main Thm}, we have
$t\geq r(e-r(n-r))-\binom{n}{3}$, which yields that $g(n)\leq \binom{n}{3}=O(n^3)$.
On the other hand, when $9\mid n$, 
let $\mathbf{k}=(\frac{n}{9},\frac{n}{9},\frac{n}{9})$.
Consider the graph $F_1^{\mathbf{k}}$ and its greedy partition $P_1^{\mathbf{k}}$ defined in Subsection~\ref{subsec:type1}
(when $9\nmid n$, just consider a suitable $n$-vertex subgraph of $F_1^{\mathbf{k}}$ for $\mathbf{k}=(\lceil\frac{n}{9}\rceil,\lceil\frac{n}{9}\rceil,\lceil\frac{n}{9}\rceil)$).
It follows that
$r(P_1^{\mathbf{k}})(e(F_1^{\mathbf{k}})-r(P_1^{\mathbf{k}})(n-r(P_1^{\mathbf{k}})))-t(P_1^{\mathbf{k}})=k_1k_2k_3$, which implies that $g(n) = \Omega(n^3)$, as desired.

Finally, it is worth noting that, although the counterexamples in Section~\ref{sec:counterexamples} disprove the assertion of Conjecture~\ref{Conj1.2} on $t(G)$, they do not violate the desired inequality $t_e(G) \ge e - r(n - r)$.  
As a consequence of Theorem~\ref{Thm:Main Thm}, when $\omega(P) < r$, applying the lemma of Huang and Shi~\cite{HS2014} that $t_e(G) \ge t(G)/r(P)$, we see that the lower bound on $t_e$ in Conjecture~\ref{Conj1.2} holds.  
It remains an interesting open question whether $t_e(G) \ge e - r(n - r)$ holds in general.

   \small
   \bibliographystyle{unsrt}

\begin{thebibliography}{99}

         \bibitem{BHKNW}
         J. Balogh, J. He, R. A. Krueger, T. Nguyen, and M. C. Wigal,
         \newblock{Clique covers and decompositions of cliques of graphs},
         \newblock{arXiv:2412.05522}, 2024.


         \bibitem{B1976}
         B. Bollob\'as,
         On complete subgraphs of different orders,
         {\it Math. Proc. Cambridge Philos. Soc.}, \textbf{79} (1976), 19--24.


         \bibitem{Chung}
         F. R. K. Chung,
         On the decomposition of graphs,
         {\it SIAM J. Algebraic Discrete Methods}, \textbf{2}(1) (1981), 1--12.


         \bibitem{Erd1971}
         P. Erd\H os,
         Some unsolved problems in graph theory and combinatorial analysis, in :
         {\it Combinatorial Mathematics and its Applications (Proc. Conf., Oxford, 1969)}, Academic Press, London-New York, 1971, 97--109.


         \bibitem{EGP}
         P. Erd{\H o}s, A. W. Goodman, and L. P{\'o}sa,
         The representation of a graph by set intersections,
         {\it Canad. J. Math.}, \textbf{18} (1966), 106--112.


         \bibitem{Gyori}
         E. Gy\H{o}ri,
         Edge-disjoint cliques in graphs, in: 
         {\it Sets, graphs and numbers (Budapest 1991)}, 
         Colloq. Math. Soc. J\'anos Bolyai,  \textbf{60}, North-Holland, Amsterdam, 1992, 357--363.


         \bibitem{GK2017}
         E. Gy\H ori, and B. Keszegh,
         On the number of edge-disjoint triangles in $K_4$-free graphs,
         {\it Combinatorica}, \textbf{37}(6) (2017), 1113--1124.


         \bibitem{GK1979}
         E. Gy\H ori, and A. V. Kostochka,
         On a problem of G. O. H. Katona and T. Tarj\'an,
         {\it Acta Math. Acad. Sci. Hungar.}, \textbf{34}(3-4) (1979), 321--327.


         \bibitem{HS2014}
         S. Huang, and L. Shi,
         Packing triangles in $K_4$-free graphs,
         {\it Graphs and Combinatorics}, \textbf{30}(3) (2014), 627--632.

         \bibitem{H05}
         N. H. Hoi, 
         On the problems of edge disjoint cliques in graphs (Master's thesis), 
         E\"otv\"os Lor\'and University, Hungary, 2005.

         \bibitem{Kahn}
         J. Kahn,
         Proof of a conjecture of Katona and Tarj\'{a}n,
         {\it Period. Math. Hungar.}, \textbf{12}(1) (1981), 81--82.


         \bibitem{Turan}
         P. Tur\'an,
         On an extremal problem in graph theory,
         {\it Matematikai \'es Fizikai Lapok}, \textbf{48} (1941), 436--452.


         \bibitem{Tuza}
         Z. Tuza,
         Unsolved Combinatorial Problems,
         Part I, BRICS Lecture Series LS-01-1, 2001.
         
	\end{thebibliography}

\normalsize
\section*{Appendix A: Justification of the inequality~\eqref{Equ:t(G)>r(e-r(n-r))-w(P)}}\label{appendix: equality proof}

Starting from inequality~\eqref{Equ:t(G)>r(e-r(n-r))-w(P)},
it suffices to show that
$$F_0(G;P)-(r-3)M_0(G;P)-\frac{a}{2}(r-1)(r-2)+a= r(e-r(n-r)).$$
Since $n=3a+2b+c$ and $r=a+b+c$, we have $3a+b=n-r+a.$
From the definition of $M_0(G;P)$ in~\eqref{Equ:defM0}, we obtain that
\begin{align}\label{Equ:defM0'}
     M_0(G;P)=&2\sum_{1\leq i<j \leq r} e_{ij} - 4n(r-1) + 8\binom r 2 -a(r-1)
             = 2e+2(3a+b)(r-2)-(r-1)(4n-4r+a)\nonumber\\
             =&2e + 2(n-r+a)(r-2) -(r-1)(4n-4r+a) = 2e -2(n-r)r +a(r-3).
\end{align}
\noindent
Similarly, by the definition of $F_0(G;P)$ in~\eqref{Equ:defF0}, we have
\begin{align}\label{Equ:defF0'}
    F_0(G;P)=&\sum_{1\leq i<j<k \leq r}3(e_{ijk}-3(|T_i|+|T_j|+|T_k|-3))
            =3\sum_{1\leq i<j<k \leq r}e_{ijk} - 9n\binom{r-1}{2} + 27 \binom{r}{3}\nonumber\\
            =&3(e-3a-b)(r-2) + 3(3a+b) \binom {r-1} 2 - \frac 9 2 (n-r)(r-1)(r-2)\nonumber\\[4pt]
    =&3e(r-2) - 3(n-r+a)(r-2) + \frac 3 2 (n-r+a) (r-1)(r-2) - \frac 9 2 (n-r)(r-1)(r-2)\nonumber\\[4pt]
    =&3e(r-2) - 3(n-r)r(r-2) + \frac{3}{2} a(r-2)(r-3).
\end{align}
Combining~\eqref{Equ:defM0'} and~\eqref{Equ:defF0'}, we obtain that
\begin{align*}
    &F_0(G;P)-(r-3)M_0(G;P)-\frac{a}{2}(r-1)(r-2)+a\\
    =&3e(r-2) - 3(n-r)r(r-2) + \frac 3 2 a (r-2)(r-3)  -
     (r-3)[2e - 2(n-r)r + a(r-3)] -\frac{a}{2}r(r-3)\\
    =&er - (n-r)r^2 + \frac a 2 (r-3) [3(r-2) - 2(r-3) -r] = r(e-r(n-r)),
\end{align*}
which completes the proof of~\eqref{Equ:t(G)>r(e-r(n-r))-w(P)}.
\QED

\section*{Appendix B: Proof of Claim~\ref{Clm:BaseCase}}\label{appendix:claimproof}

    Note that, for $r=3$, we have $F_0(G;P)=3(e(G)-3(v(G)-3))$.
    We point out that it is always the case that $\omega(P)\in \{0,1\}$. 
    The claim reduces to proving
    \begin{align}\label{Equ: When r=3, t>F+C-w}
        t(G)\geq F_0(G;P)+M_2(G;P)-M_0(G;P)-\omega(P).
    \end{align}
    Let $P=\{T_1,T_2,T_3\}$ and let $a,b,c$ be the number of cliques of size $3,2$ and $1$, respectively.

    Suppose $|T_1|\le 2,$ i.e., $a=0$.
    By the definition of greedy partition, we have
    $t(G)=M_2(G;P)=\omega(P)=0,$
    and $M_0(G;P)=\sum_{1\le i<j\in[3]}2(e_{ij}-2(|T_i|+|T_j|-2))-2a=2(e+v(G)-3)-8v(G)+24.$
    Thus inequality~\eqref{Equ: When r=3, t>F+C-w} is equivalent to
    $e(G)-3v(G)+9\le 0,$
    which holds by straightforward case analysis.

    We verify the remaining subcases for $r=3$ with $|T_1|= 3$ (i.e., $a\ge 1$) via computer assistance (see the program below for details).
    For every $K_4$-free graph with parameters $(a, b, c)$ where $a\ge 1$, 
    our program verifies whether or not the input graph satisfies the inequality
    \begin{align}\label{Equ: When r=3, t>F+M2-M0}
    t(G)\geq F_0(G;P)+M_2(G;P)-M_0(G;P).
    \end{align}
    If it satisfies, then evidently it satisfies \eqref{Equ: When r=3, t>F+C-w};
    otherwise, the program identifies graphs that do not satisfy this inequality (called {\it  Counterexamples to~\eqref{Equ: When r=3, t>F+M2-M0}}), along with the parameters $\{t(G), M_2(G;P), e(G)\}$.
    These are summarized in Table~\ref{tab:program}.
    Note that $\omega(P)=1$ holds for all Counterexamples to~\eqref{Equ: When r=3, t>F+M2-M0}.
    After thorough calculations, inequality~\eqref{Equ: When r=3, t>F+C-w} holds for all Counterexamples to~\eqref{Equ: When r=3, t>F+M2-M0}, thus Claim~\ref{Clm:BaseCase} holds and we are done. \QED
\bigskip

\begin{table}[htbp]
\centering
\renewcommand{\arraystretch}{1.6}
\begin{NiceTabular}{c|c|c|c|>{\centering\arraybackslash}m{5cm}}
\toprule
$(a, b, c)$ & $M_0(G;P)$ & $F_0(G;P)$ & Precise expression of Inequality \eqref{Equ: When r=3, t>F+M2-M0} & Counterexamples to~\eqref{Equ: When r=3, t>F+M2-M0} and parameters $\{t(G), M_2(G;P), e(G)\}$\\
\midrule

$(3,0,0)$ & $2e-36$ & $3e-54$ & $t(G)\ge M_2(G;P)+e(G)-18$ & \!\!\!$F_1: \{11, 8, 22\}$ \quad $F_2: \{14, 10, 23\}$ \quad $F_3: \{13, 10, 22\}$\\
\midrule

$(2,1,0)$ & $2e-30$ & $3e-45$ & $t(G)\ge M_2(G;P)+e(G)-15$ & \!\!\!$F_4: \{10, 8, 18\}$ \\
\midrule

$(2,0,1)$ & $2e-24$ & $3e-36$ & $t(G)\ge M_2(G;P)+e(G)-12$ & $\emptyset$ \\
\midrule

$(1,2,0)$ & $2e-24$ & $3e-36$ & $t(G)\ge M_2(G;P)+e(G)-12$ & $\emptyset$ \\
\midrule

$(1,1,1)$ & $2e-18$ & $3e-27$ & \!\!\!\!$t(G)\ge M_2(G;P)+e(G)-9$ & $\emptyset$ \\
\midrule

$(1,0,2)$ & $2e-12$ & $3e-18$ & \!\!\!\!$t(G)\ge M_2(G;P)+e(G)-6$ & $\emptyset$ \\
\bottomrule

\CodeAfter
\tikz \draw [white, line width=0.7pt] 
    (1-|1) -- (last-|1)    
    (1-|6) -- (last-|6);  
\end{NiceTabular}
\caption{All the subcases for $r=3$ with $|T_1|= 3$ in Claim~\ref{Clm:BaseCase}.}
\label{tab:program}
\end{table}
\clearpage

\noindent\textbf{The program for $r=3$ with $|T_1|= 3$ in Claim~\ref{Clm:BaseCase}.}
\lstset{language=Python}
\begin{lstlisting}
import networkx as nx
import itertools
import os
import matplotlib.pyplot as plt

class GraphSolution:
    # Initialize the graphs: we first fix the structure between two of {T_1,T_2,T_3} (self.isomorphism_list) and then traverse all the remaining edges (self.possible_add_edge1/2).
    # For different types, the difference between the inequalities is stored in self.constant, i.e., if t(G)<m_2(G)+e(G)-self.constant.
    def __init__(self):
        self.graph = None
        self.isomorphism_list = []
        self.possible_add_edge1 = []
        self.possible_add_edge2 = []
        self.constant = 0
        self.ori_tri_list = []            # The number of triangles in {G[T_1],G[T_2]}, {G[T_1],G[T_3]}, and {G[T_2],G[T_3]}.
        self.subgraphs = []               # The induced subgraphs on vertex sets T_1UT_2, T_1UT_3, and T_2UT_3.
        self.save_root = './graph_result'

    # Verify if the graph G is k_4-free by checking the number of edges of subgraphs induced on any 4 vertices of G.
    def is_k4_free(self, G=None):
        if G is None:
            G = self.graph
        for nodes in itertools.combinations(G.nodes, 4):
            subgraph = G.subgraph(nodes)
            if subgraph.number_of_edges() == 6:
                return False
        return True

    # Calculate m_2(G) by adding up m_2(G[T_1UT_2]), m_2(G[T_1UT_3]), and m_2(G[T_2UT_3]).
    def triangle_nums_bt_tri(self, G=None):
        if G is None:
            G = self.graph
        count = 0
        for i, selected_nodes in enumerate(self.subgraphs):
            subgraph = G.subgraph(selected_nodes)
            triangle_count = sum(nx.triangles(subgraph).values()) // 3 - self.ori_tri_list[i]
            count += triangle_count
        return count

    # Determine whether the graph is a new graph up to isomorphism.
    def isomorphic_list(self, graph_list, new_graph):
        i = 0
        for _,graph in enumerate(graph_list):
            GM = nx.isomorphism.GraphMatcher(graph,new_graph)
            if GM.is_isomorphic():
                break
            else:
                i += 1
        if i == len(graph_list):
            graph_list.append(new_graph)
        else:
            new_graph = None
        return graph_list, new_graph

    # Initialize graphs according to different types.
    def initialize_graph(self, graph_type):
        initializers = {
            1: self._initial_graph_1,
            2: self._initial_graph_2,
            3: self._initial_graph_3,
            4: self._initial_graph_4,
            5: self._initial_graph_5,
            6: self._initial_graph_6,
        }
        # Read graphs, and the structure between the two given cliques.
        # Traverse all the remaining edges, and select different inequalities.
        if graph_type in initializers:
            self.graph, self.isomorphism_list, self.possible_add_edge1, self.possible_add_edge2, self.constant, self.ori_tri_list, self.subgraphs = initializers[graph_type]()
        else:
            raise ValueError("Invalid graph type")

    # Add edges and verify if the inequalities hold.
    def forward(self):
        os.makedirs(self.save_root, exist_ok=True)
        graph_list = []

        count = 0
        # Read subgraphs induced on the vertex set of the two fixed cliques (G_1).
        for i, isomorphism in enumerate(self.isomorphism_list):
            print(f'Starting to add {i+1} isomorphism')
            # Traverse all the subgraphs induced on the vertex set of one of the undetermined structure between two cliques (G_2).
            for j in range(1, len(self.possible_add_edge1) + 1):
                combinations1 = list(itertools.combinations(self.possible_add_edge1, j))
                # Traverse all the subgraphs induced on the vertex set of the other undetermined structure between two cliques (G_3).
                for k in range(1, len(self.possible_add_edge2) + 1):
                    combinations2 = list(itertools.combinations(self.possible_add_edge2, k))
                    for combo1 in combinations1:
                        for combo2 in combinations2:
                            G_new = self.graph.copy()
                            G_new.add_edges_from(isomorphism)     # Add edges for subgraph G_1.
                            G_new.add_edges_from(combo1)          # Add edges for subgraph G_2.
                            G_new.add_edges_from(combo2)          # Add edges for subgraph G_3.
                            # Verify if the graph G is k_4-free
                            if self.is_k4_free(G_new):
                                num_edge = G_new.number_of_edges()                              # Calculate the number of edges of G.
                                triangle_count = sum(nx.triangles(G_new).values()) // 3         # Calculate the number of triangles of G.
                                m2_num = self.triangle_nums_bt_tri(G_new)                       # Calculate m_2(G).
                                # Verify if t(G)<m_2(G)+e(G)-self.constant
                                if triangle_count < m2_num + num_edge - self.constant:
                                    graph_list, new_graph = self.isomorphic_list(graph_list, G_new)  # Determine whether the graph is a new graph.
                                    if new_graph is not None:
                                        plt.figure(figsize=(8, 6))
                                        pos = nx.spring_layout(G_new)  # Select a layout algorithm
                                        nx.draw(G_new, with_labels=True, node_color='lightblue', edge_color='gray')
                                        filename = os.path.join(self.save_root, f'graph_edge={num_edge}_triangle={triangle_count}_m2={m2_num}_{count}.png')
                                        count += 1
                                        plt.savefig(filename)    # Save the graph.
                                        plt.close()

    # Type I (a=3, and T_1={1,2,3}, T_2={4,5,6}, T_3={7,8,9}).
    def _initial_graph_1(self):
        G = nx.Graph()
        G.add_nodes_from([1, 2, 3, 4, 5, 6, 7, 8, 9])
        G.add_edges_from([(1, 2), (1, 3), (2, 3), (4, 5), (4, 6), (5, 6), (7, 8), (7, 9), (8, 9)])
        # Different K_4-free graphs on vertex set T_1UT_2.
        isomorphism_list = [
            [],
            [(1,4)], [(1,4),(1,5)], [(1,4),(2,5)], [(1,4),(1,5),(3,5)], [(1,4),(1,5),(2,6)],
            [(1,4),(2,5),(3,6)], [(1,4),(1,5),(2,5),(2,6)], [(1,4),(1,5),(3,5),(2,6)],
            [(1,4),(1,5),(3,4),(2,6)],
            [(1,4),(1,5),(3,4),(2,5),(2,6)],
            [(1,4),(1,5),(3,4),(2,5),(2,6),(3,6)]
        ]
        # Add possible edges between T_1 and T_3.
        possible_add_edge1 = list(set(itertools.combinations([1, 2, 3, 7, 8, 9], 2)) - set(G.edges))
        # Add possible edges between T_2 and T_3.
        possible_add_edge2 = list(set(itertools.combinations([4, 5, 6, 7, 8, 9], 2)) - set(G.edges))
        # self_constant=18, i.e., check if t(G)<m_2(G)+e(G)-18
        return G, isomorphism_list, possible_add_edge1, possible_add_edge2, 18, [2,2,2], [[1, 2, 3, 4, 5, 6], [1, 2, 3, 7, 8, 9], [4, 5, 6, 7, 8, 9]]

    # Type II (a=2, b=1, and T_1={1,2,3}, T_2={4,5,6}, T_3={7,8}).
    def _initial_graph_2(self):
        G = nx.Graph()
        G.add_nodes_from([1, 2, 3, 4, 5, 6, 7, 8])
        G.add_edges_from([(1, 2), (1, 3), (2, 3), (4, 5), (4, 6), (5, 6), (7, 8)])
        isomorphism_list = [
            [],
            [(1,4)], [(1,4),(1,5)], [(1,4),(2,5)], [(1,4),(1,5),(3,5)], [(1,4),(1,5),(2,6)],
            [(1,4),(2,5),(3,6)], [(1,4),(1,5),(2,5),(2,6)], [(1,4),(1,5),(3,5),(2,6)],
            [(1,4),(1,5),(3,4),(2,6)],
            [(1,4),(1,5),(3,4),(2,5),(2,6)],
            [(1,4),(1,5),(3,4),(2,5),(2,6),(3,6)]
        ]
        possible_add_edge1 = list(set(itertools.combinations([1, 2, 3, 7, 8], 2)) - set(G.edges))
        possible_add_edge2 = list(set(itertools.combinations([4, 5, 6, 7, 8], 2)) - set(G.edges))
        # self_constant=15, i.e., check if t(G)<m_2(G)+e(G)-15
        return G, isomorphism_list, possible_add_edge1, possible_add_edge2, 15, [2,1,1], [[1, 2, 3, 4, 5, 6], [1, 2, 3, 7, 8], [4, 5, 6, 7, 8]]

    # Type III (a=2, c=1, and T_1={1,2,3}, T_2={4,5,6}, T_3={7}).
    def _initial_graph_3(self):
        G = nx.Graph()
        G.add_nodes_from([1, 2, 3, 4, 5, 6, 7])
        G.add_edges_from([(1, 2), (1, 3), (2, 3), (4, 5), (4, 6), (5, 6)])
        isomorphism_list = [
            [],
            [(1,4)], [(1,4),(1,5)], [(1,4),(2,5)], [(1,4),(1,5),(3,5)], [(1,4),(1,5),(2,6)],
            [(1,4),(2,5),(3,6)], [(1,4),(1,5),(2,5),(2,6)], [(1,4),(1,5),(3,5),(2,6)],
            [(1,4),(1,5),(3,4),(2,6)],
            [(1,4),(1,5),(3,4),(2,5),(2,6)],
            [(1,4),(1,5),(3,4),(2,5),(2,6),(3,6)]
        ]
        possible_add_edge1 = list(set(itertools.combinations([1, 2, 3, 7], 2)) - set(G.edges))
        possible_add_edge2 = list(set(itertools.combinations([4, 5, 6, 7], 2)) - set(G.edges))
        # self_constant=12, i.e., check if t(G)<m_2(G)+e(G)-12.
        return G, isomorphism_list, possible_add_edge1, possible_add_edge2, 12, [2,1,1], [[1, 2, 3, 4, 5, 6], [1, 2, 3, 7], [4, 5, 6, 7]]

    # Type IV (a=1, b=2, and T_1={1,2,3}, T_2={4,5}, T_3={6,7}).
    def _initial_graph_4(self):
        G = nx.Graph()
        G.add_nodes_from([1, 2, 3, 4, 5, 6, 7])
        G.add_edges_from([(1, 2), (1, 3), (2, 3), (4, 5), (6, 7)])
        # Different K_3-free graphs on vertex set T_2UT_3.
        isomorphism_list = [[], [(4,6)], [(4,6),(5,7)]]
        # Add possible edges between T_1 and T_2.
        possible_add_edge1 = list(set(itertools.combinations([1, 2, 3, 4, 5], 2)) - set(G.edges))
        # Add possible edges between T_1 and T_3.
        possible_add_edge2 = list(set(itertools.combinations([1, 2, 3, 6, 7], 2)) - set(G.edges))
        # self_constant=12, i.e., check if t(G)<m_2(G)+e(G)-12.
        return G, isomorphism_list, possible_add_edge1, possible_add_edge2, 12, [1,1,0], [[1, 2, 3, 4, 5], [1, 2, 3, 6, 7], [4, 5, 6, 7]]

    # Type V (a=1, b=1, c=1, and T_1={1,2,3}, T_2={4,5}, T_3={6}).
    def _initial_graph_5(self):
        G = nx.Graph()
        G.add_nodes_from([1, 2, 3, 4, 5, 6])
        G.add_edges_from([(1, 2), (1, 3), (2, 3), (4, 5)])
        # Different K_3-free graphs on vertex set T_2UT_3.
        isomorphism_list = [[], [(4,6)]]
        possible_add_edge1 = list(set(itertools.combinations([1, 2, 3, 4, 5], 2)) - set(G.edges))
        possible_add_edge2 = list(set(itertools.combinations([1, 2, 3, 6], 2)) - set(G.edges))
        # self_constant=9, i.e., check if t(G)<m_2(G)+e(G)-9.
        return G, isomorphism_list, possible_add_edge1, possible_add_edge2, 9, [1,1,0], [[1, 2, 3, 4, 5], [1, 2, 3, 6], [4, 5, 6]]

    # Type VI (a=1, c=2, and T_1={1,2,3}, T_2={4}, T_3={5}).
    def _initial_graph_6(self):
        G = nx.Graph()
        G.add_nodes_from([1, 2, 3, 4, 5])
        G.add_edges_from([(1, 2), (1, 3), (2, 3)])
        # Different K_2-free graphs on vertex set T_2UT_3.
        isomorphism_list = [[]]
        possible_add_edge1 = list(set(itertools.combinations([1, 2, 3, 4], 2)) - set(G.edges))
        possible_add_edge2 = list(set(itertools.combinations([1, 2, 3, 5], 2)) - set(G.edges))
        # self_constant=6, i.e., check if t(G)<m_2(G)+e(G)-6.
        return G, isomorphism_list, possible_add_edge1, possible_add_edge2, 6, [1,1,0], [[1, 2, 3, 4], [1, 2, 3, 5], [4, 5]]

# Run all the types and output the counterexamples
graph_solution = GraphSolution()
for i in range(6,0,-1):
    print(i)
    graph_solution.initialize_graph(i)
    graph_solution.forward()
\end{lstlisting}

\end{document}